\documentclass[11pt, a4paper]{amsart}
\usepackage{amsmath, amsthm, amsfonts, amssymb, mathtools}
\usepackage{a4wide}
\usepackage[latin1]{inputenc}
\usepackage{graphicx}
\usepackage{enumitem}
\usepackage{bbm}
\newcommand{\SL}{\mathrm{SL}}
\newcommand{\N}{\mathbb{N}}
\newcommand{\Z}{\mathbb{Z}}
\newcommand{\R}{\mathbb{R}}
\newcommand{\C}{\mathbb{C}}
\newcommand{\Q}{\mathbb{Q}}
\renewcommand{\H}{\mathbb{H}}
\DeclareMathOperator{\tr}{tr}
\DeclareMathOperator{\sgn}{sgn}
\renewcommand{\Im}{\mathrm{Im}}
\renewcommand{\Re}{\mathrm{Re}}

\numberwithin{equation}{section}
\newtheorem{theorem}{Theorem}
\numberwithin{theorem}{section}
\newtheorem{lemma}[theorem]{Lemma}
\newtheorem{proposition}[theorem]{Proposition} 

\theoremstyle{definition} 
\newtheorem{definition}[theorem]{Definition}
\newtheorem{example}[theorem]{Example}
\newtheorem{remark}[theorem]{Remark}

\author{Steffen L\"obrich and Markus Schwagenscheidt}

\address{Korteweg-de Vries Institute for Mathematics, University of Amsterdam, Science Park 105-107, 1098 XG Amsterdam, The Netherlands}
\email{s.loebrich@uva.nl}

\address{ETH Z\"urich Mathematics Dept., R\"amistrasse 101, CH-8092 Z\"urich, Switzerland}
\email{mschwagen@ethz.ch}

\title{Locally harmonic Maass forms and periods of meromorphic modular forms}

\allowdisplaybreaks

\thanks{The work of the first author is supported by ERC starting grant H2020 ERC StG \#640159. The second author is supported by SNF project 200021\_185014. We thank Kathrin Bringmann for helpful discussions. Furthermore, we thank Joshua Males for useful remarks on an earlier draft of this paper.}

\begin{document}

\begin{abstract}
We investigate a new family of locally harmonic Maass forms which correspond to periods of modular forms. They transform like negative weight modular forms and are harmonic apart from jump singularities along infinite geodesics. Our main result is an explicit splitting of the new locally harmonic Maass forms into a harmonic part and a locally polynomial part 
that captures the jump singularities. As an application, we obtain finite rational formulas for suitable linear combinations of periods of meromorphic modular forms associated to positive definite binary quadratic forms.
\end{abstract}

\maketitle

\section{Introduction and statement of results}

\subsection{Locally harmonic Maass forms and cycle integrals} In the early 2000s, Zwegers \cite{zwegers} made the groundbreaking discovery that Ramanujan's mock theta functions, whose precise automorphic nature had been a long-standing conundrum, could be viewed as the holomorphic parts of \emph{harmonic weak Maass forms} of weight $1/2$ whose shadows are unary theta functions of weight $3/2$. The theory of harmonic weak Maass forms was developed systematically around the same time by Bruinier and Funke \cite{bruinierfunke}. Since then, it has become a vital area of research in number theory and has found many fascinating applications, for example to the singular theta correspondence and Borcherds products \cite{bruinierhabil, bruinierfunke}, the partition function and its variants \cite{ahlgrenandersen, bruinieronoalgebraic}, special values of $L$-functions of elliptic curves \cite{bruinieronoheegner}, and CM values of higher Green functions \cite{bringmannkanevonpippich, bruinierehlenyang, li}.
 
More recently, Bringmann, Kane, and Kohnen \cite{bringmannkanekohnen} constructed a new type of harmonic weak Maass forms which have jump singularities along certain geodesics in the upper half-plane $\H$, and hence are called \emph{locally harmonic Maass forms} (see also \cite{brikavia, crawford, crawfordfunke, hoevel}). Specifically, for $k \in \Z$ with $k \geq 2$ the authors of \cite{bringmannkanekohnen} associated to each non-square discriminant $D > 0$ the function ($\tau = u+iv \in \H$)
\[
\mathcal{F}_{1-k,D}(\tau) := \frac{(-1)^k D^{\frac{1}{2}-k}}{\binom{2k-2}{k-1}\pi}\sum_{Q = [a,b,c]\in \mathcal{Q}_{D}}\sgn\left(a|\tau|^{2}+bu+c\right)Q(\tau,1)^{k-1}\psi\left(\frac{D^{2}v}{|Q(\tau,1)|^{2}} \right),
\]
where $\mathcal{Q}_{D}$ denotes the set of all (positive definite if $D < 0$) integral binary quadratic forms $Q(x,y) = ax^{2}+bxy + cy^{2}$ of discriminant $D = b^{2}-4ac$, and $\psi(v) := \frac{1}{2}\int_{0}^{v}t^{k-\frac{3}{2}} (1-t)^{-\frac{1}{2}}dt$ is a special value of the incomplete $\beta$-function. The function $\mathcal{F}_{1-k,D}(\tau)$ transforms like a modular form of negative even weight $2-2k$ for $\Gamma := \SL_{2}(\Z)$, it is bounded at the cusp, and it is harmonic on $\H$ up to jump singularities along the exceptional set
\[
E_{D} :=\{\tau=u+iv \in \H \,:\, a|\tau|^{2}+bu+c = 0 \, , \, [a,b,c] \in \mathcal{Q}_{D}\}.
\]
Note that $E_{D}$ is a union of semi-circles centered at the real line if $D > 0$ is not a square. A special feature of the locally harmonic Maass form $\mathcal{F}_{1-k,D}(\tau)$ is the fact that its images under the two differential operators
\[
\xi_{2-2k} := 2iv^{2-2k}\overline{\frac{\partial}{\partial \overline{\tau}}}, \qquad \mathcal{D}^{2k-1} := \left(\frac{1}{2\pi i}\frac{\partial}{\partial \tau} \right)^{2k-1},
\]
are non-zero multiples of the weight $2k$ cusp form
\begin{align}\label{eq fkD}
f_{k,D}(\tau) := \frac{D^{k-\frac{1}{2}}}{\pi}\sum_{Q \in \mathcal{Q}_{D}}Q(\tau,1)^{-k}.
\end{align}
In contrast, it is impossible for a harmonic weak Maass form to map to a non-zero multiple of the same cusp form under both operators $\xi_{2-2k}$ and $\mathcal{D}^{2k-1}$. The cusp form $f_{k,D}(\tau)$ can be characterized by the fact that the Petersson inner product $\langle f,f_{k,D}\rangle$ of a cusp form $f(\tau)$ of weight $2k$ with $f_{k,D}(\tau)$ is a certain multiple of the $D$-th trace of cycle integrals
\[
\tr_{D}(f) := \sum_{Q \in \mathcal{Q}_{D}/\Gamma}\int_{\Gamma_{Q}\backslash S_{Q}}f(z)Q(z,1)^{k-1}dz
\] 
of $f(\tau)$. Here $\Gamma_{Q}$ denotes the stabilizer of $Q$ in $\Gamma$ and $S_{Q}$ for $Q = [a,b,c]$ is the semi-circle consisting of all $\tau = u+iv \in \H$ with $a|\tau|^{2}+bu+c = 0$. In this sense, the locally harmonic Maass forms $\mathcal{F}_{1-k,D}(\tau)$ correspond to the (traces of) cycle integrals of cusp forms.

\subsection{Locally harmonic Maass forms and periods} In the present article we are concerned with the construction of a new family of locally harmonic Maass forms $\mathcal{H}_{1-k,n}(\tau)$, for $k \in \Z$ with $k \geq 2$ and integral $0 \leq n \leq 2k-2$, which correspond to the \emph{periods} of cusp forms, in a sense which will become apparent in a moment. In order to give the definition of $\mathcal{H}_{1-k,n}(\tau)$, for $k,\ell \in \Z$ with $k\geq 2$ we consider Petersson's Poincar\'e series ($z, \tau = u+iv \in \H$)
\begin{align}\label{eq Peterssons Poincare series}
H_{k,\ell}(z,\tau) := \sum_{M \in \Gamma}v^{k+\ell}(z-\tau)^{\ell-k}(z-\overline{\tau})^{-\ell-k}\Bigl|_{-2\ell,\tau}M,
\end{align}
with the usual slash operator applied in the $\tau$-variable. It transforms like a modular form of weight $2k$ in $z$ and of weight $-2\ell$ in $\tau$. We will be particularly interested in the case $\ell = k-1$. Then the function $H_{k,k-1}(z,\tau)$ is a meromorphic modular form of weight $2k$ in $z$ which has poles precisely at the $\Gamma$-translates of $\tau$ and decays like a cusp form towards $i\infty$. Furthermore, it transforms like a modular form of weight $2-2k$ in $\tau$ and is harmonic on $\H \setminus \Gamma z$. We can now make the following definition.

\begin{definition}
	For $k \in \Z$ with $k \geq 2$ and integral $0 \leq n \leq 2k-2$ we define the function
	\begin{align}\label{definition curlyH}
	\mathcal{H}_{1-k,n}(\tau):= \frac{(2i)^{2k-2}}{2\pi}\int_{0}^{\infty}H_{k,k-1}(iy,\tau)y^{n}dy.
	\end{align}
	If $\tau \in \bigcup_{M \in \Gamma}M(i \R_{+})$ then the integral is defined using the Cauchy principal value (see \cite{ls}, Section~3.5).
\end{definition}

In other words, the function $\mathcal{H}_{1-k,n}(\tau)$ is essentially the $n$-th period of the meromorphic modular form $z \mapsto H_{k,k-1}(z,\tau)$. In \cite{ls} we showed that the locally harmonic Maass form $\mathcal{F}_{1-k,D}(\tau)$ from \cite{bringmannkanekohnen} can be viewed as the $D$-th trace of cycle integrals of the function $z \mapsto H_{k,k-1}(z,\tau)$, which inspired the above definition of $\mathcal{H}_{1-k,n}(\tau)$. 

It is clear from the construction that $\mathcal{H}_{1-k,n}(\tau)$ transforms like a modular form of weight $2-2k$ for $\Gamma$ and is harmonic on $\H \setminus\bigcup_{M \in \Gamma}M(i \R_{+})$. Further, it follows from Theorem~\ref{theorem splitting} below that $\mathcal{H}_{1-k,n}(\tau)$ is of moderate growth at the cusp and has jump singularities along the $\Gamma$-translates of the positive imaginary axis $i\R_{+}$. In particular, the function $\mathcal{H}_{1-k,n}(\tau)$ defines a locally harmonic Maass form of weight $2-2k$ for $\Gamma$ with exceptional set $E_{1} = \bigcup_{M \in \Gamma}M(i\R_{+})$ in the sense of \cite{bringmannkanekohnen}.

\begin{figure}[h]\label{figure}
  \centering
    \includegraphics[width=0.95\textwidth]{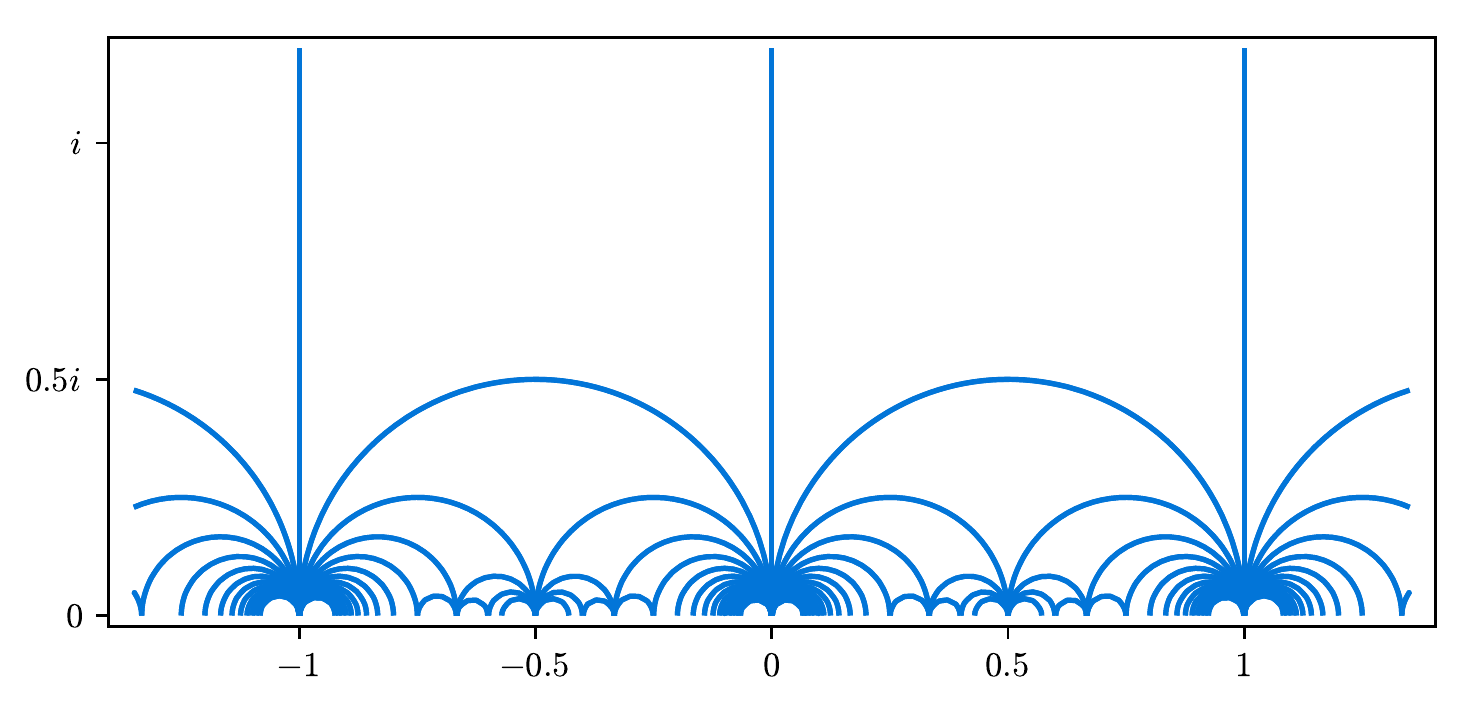}
		\caption{The set $E_1$ of singularities of $\mathcal{H}_{1-k,n}$.}
\end{figure}

We first determine the images of $\mathcal{H}_{1-k,n}(\tau)$ under the differential operators $\xi_{2-2k}$ and $\mathcal{D}^{2k-1}$. To state the result, we let $E_{2k}(\tau)$ be the usual normalized Eisenstein series of weight $2k$, and for $0 \leq n \leq 2k-2$ we let $R_{n}(\tau)$ be the weight $2k$ cusp form which is characterized by the fact that its Petersson inner product $\langle f,R_{n} \rangle$ with a cusp form $f(\tau)$ of weight $2k$ equals the $n$-th period
\begin{align}\label{definition periods}
r_{n}(f) := \int_{0}^{\infty}f(iy)y^{n}dy
\end{align}
of $f(\tau)$. Note that the periods satisfy the symmetry $r_n(f) = (-1)^k r_{2k-2-n}(f)$. The cusp form $R_n(\tau)$ is related to $\mathcal{H}_{1-k,n}(\tau)$ in the following way.

\begin{proposition}\label{proposition diffops} For $\tau \in \H \setminus \bigcup_{M \in \Gamma}M(i\R_{+})$ we have
\begin{align*}
\xi_{2-2k}\mathcal{H}_{1-k,n}(\tau) &= -R_{n}(\tau), \\
\mathcal{D}^{2k-1}\mathcal{H}_{1-k,n}(\tau) &= (-1)^{n+1}\frac{(2k-2)!}{(4\pi)^{2k-1}}R_{n}(\tau)  -\left((-1)^k\delta_{n = 0}+\delta_{n=2k-2}\right)\frac{(2k-2)!}{(2\pi)^{2k-1}}E_{2k}(\tau).
\end{align*}
\end{proposition}

We refer the reader to Section~\ref{section proof proposition diffops} for the proof of Proposition~\ref{proposition diffops}. The above proposition implies that $\mathcal{H}_{1-k,n}(\tau)$ can be written as a sum of the holomorphic and non-holomorphic Eichler integrals of $R_{n}(\tau)$ (and of $E_{2k}(\tau)$ if $n = 0$ or $n = 2k-2$) and a locally polynomial part, which captures the singularities of $\mathcal{H}_{1-k,n}(\tau)$. Recall that the holomorphic and non-holomorphic Eichler integrals of a cusp form $f(\tau) = \sum_{n=1}^{\infty}c_{f}(n)e(n\tau)$ of weight $2k$ (with $e(x):= e^{2\pi i x}$ for $x \in \C$) are defined by
\begin{align}\label{eq eichler integrals}
\begin{split}
\mathcal{E}_{f}(\tau) &:= \frac{(-2\pi i)^{2k-1}}{(2k-2)!}\int_{\tau}^{i\infty}f(z)(z-\tau)^{2k-2}dz= \sum_{n\geq 1}\frac{c_{f}(n)}{n^{2k-1}}e(n\tau), \\
f^{*}(\tau) &:= (-2i)^{1-2k}\int_{-\overline{\tau}}^{i\infty}\overline{f(-\overline{z})}(z+\tau)^{2k-2}dz = -\sum_{n\geq 1}\frac{\overline{c_{f}(n)}}{(4\pi n)^{2k-1}}\Gamma(2k-1, 4\pi nv)e(-n\tau),
\end{split}
\end{align}
where $\Gamma(s,x):=\int_{x}^\infty e^{-t} t^{s-1}dt$ is the incomplete Gamma function. They satisfy
\begin{align}\label{eq diffops eichler}
\xi_{2-2k}f^{*}(\tau) = f(\tau), \qquad \mathcal{D}^{2k-1}f^{*}(\tau) = 0, \qquad \xi_{2-2k}\mathcal{E}_{f}(\tau) = 0, \qquad \mathcal{D}^{2k-1}\mathcal{E}_{f}(\tau) = f(\tau).
\end{align}
By using the series expansions on the right-hand sides of \eqref{eq eichler integrals}, we can extend the definitions of the Eichler integrals to the Eisenstein series $E_{2k}(\tau)$. The main result of this work is the following splitting of $\mathcal{H}_{1-k,n}(\tau)$.

\begin{theorem}\label{theorem splitting}
We have the decomposition
\begin{align}\begin{split}\label{split}
\mathcal{H}_{1-k,n}(\tau) &= \mathcal{P}_{1-k,n}(\tau) + (-1)^{n+1}\frac{(2k-2)!}{(4\pi )^{2k-1}}\mathcal{E}_{R_n}(\tau) - R_{n}^*(\tau) \\
 &\quad-((-1)^k\delta_{n=0}+\delta_{n=2k-2})\frac{(2k-2)!}{(2\pi)^{2k-1}}\mathcal{E}_{E_{2k}}(\tau),
 \end{split}
\end{align}
where $\mathcal{P}_{1-k,n}(\tau)$ is locally a polynomial on the connected components of $\H \setminus \bigcup_{M \in \Gamma}M(i\R_{+})$. It is explicitly given for $\tau\in\H \setminus \bigcup_{M \in \Gamma}M(i\R_{+})$ by
\begin{align*}
\mathcal{P}_{1-k,n}(\tau) &= r_n(E_{2k}) +\frac{(-i)^{n+1}}{n+1}\mathbb{B}_{n+1}(\tau)+\frac{i^{n+1}}{2k-1-n}\mathbb{B}_{2k-1-n}(\tau)
 \\
 &\quad +\frac{i^{1-n}}{2}\sum_{\substack{M = \left(\begin{smallmatrix}a & b \\ c & d \end{smallmatrix}\right)\in\Gamma \\ ac>0, \Re(M\tau)< 0}}(\tau^n -(-1)^n \tau^{2k-2-n})\Bigl |_{2-2k}M.
\end{align*}
Here $\mathbb{B}_{m}(\tau)$ is the $1$-periodic function on $\H$ which agrees with the Bernoulli polynomial $B_{m}(\tau)$ for $0 < u < 1$ and
$r_{n}(E_{2k})$ is the $n$-th period of the Eisenstein series $E_{2k}(\tau)$ given in \eqref{periodEisenstein}. \\
If there is a matrix $M\in\Gamma$ with $\Re(M\tau)=0$, then $\mathcal{P}_{1-k,n}(\tau)$ is given by the average value
\[
\lim_{\varepsilon\rightarrow 0+}\frac{1}{2}\big(\mathcal{P}_{1-k,n}(M^{-1}(M\tau - \varepsilon))+\mathcal{P}_{1-k,n}(M^{-1}(M\tau + \varepsilon))\big).
\] 
\end{theorem}

\begin{remark}Note that the sum in the second line of $\mathcal{P}_{1-k,n}(\tau)$ is finite for every $\tau\in\H$ and vanishes if $\Im(\tau) > \frac{1}{2}$. Indeed, the condition $ac > 0$ for $M = \left(\begin{smallmatrix}a & b \\ c & d \end{smallmatrix}\right) \in \Gamma$ implies that $M^{-1}(i\R_+)$ is a semi-circle of radius $\frac{1}{2ac}$ centered at the real line, and $\Re(M \tau) < 0$ means that $\tau$ lies in the bounded component of $\H\setminus M^{-1}(i\R_+)$. See also Figure 1.
\end{remark}

Our proof of the above theorem is quite different from the proof of the analogous splitting of the locally harmonic Maass form $\mathcal{F}_{1-k,D}(\tau)$ given in \cite{bringmannkanekohnen}. As the main step, we will employ an explicit formula for the period polynomial of $R_{n}(\tau)$ due to Kohnen and Zagier \cite{kohnenzagierrationalperiods} in order to directly show the modularity of the expression on the right-hand side of \eqref{split}. We refer the reader to Section~\ref{section proof theorem splitting} for the details of the proof. 

 \subsection{Periods of meromorphic modular forms} As an application of the above results, we study the periods of certain meromorphic modular forms. Namely, for a fixed positive definite quadratic form $P \in \mathcal{Q}_{d}$ of discriminant $d < 0$ we consider the function
 \begin{align}\label{definition fkP}
 f_{k,P}(z):= \frac{|d|^{k-\frac{1}{2}}}{\pi}\sum_{Q \in [P]}Q(z,1)^{-k},
 \end{align}
 where $[P]$ denotes the class of $P$ in $\mathcal{Q}_{d}/\Gamma$. We let 
 \[
 f_{k,d}(z) := \sum_{P \in \mathcal{Q}_{d}/\Gamma}f_{k,P}(z) = \frac{|d|^{k-\frac{1}{2}}}{\pi}\sum_{Q \in \mathcal{Q}_{d}}Q(z,1)^{-k}.
 \]
 Notice the analogy with the definition of the cusp form $f_{k,D}(\tau)$ in \eqref{eq fkD} for $D > 0$. The function $f_{k,P}(z)$ transforms like a modular form of weight $2k$ for $\Gamma$ and decays like a cusp form towards $i\infty$. Moreover, it is meromorphic on $\H$ and has poles of order $k$ precisely at the $\Gamma$-translates of the CM point $\tau_{P}$ defined by $P(\tau_{P},1) = 0$. 
 
 It was shown in \cite{anbs, anbms, ls} that certain linear combinations of geodesic cycle integrals of $f_{k,P}(z)$ are rational. Motivated by these results, in the present work we investigate the rationality of the periods
 \[
 r_{n}(f_{k,P}) = \int_{0}^{\infty}f_{k,P}(iy)y^{n}dy
 \]
 for $0 \leq n \leq 2k-2$, where the integral is defined using the Cauchy principal value as in \cite{ls}, Section~3.5, if a pole of $f_{k,P}$ lies on the positive imaginary axis. To get a first idea how these periods look, we consider the quadratic form $P = [1,1,1]$ of discriminant $d=-3$. In this case, there is only one class in $\mathcal{Q}_{-3}/\Gamma$, so we have $f_{k,[1,1,1]} = f_{k,-3}$. Numerical integration yields the following values of the periods of $f_{k,-3}$ for small values of $k$.
\begin{align*}
	 \renewcommand{\arraystretch}{1.2}
	\begin{array}{|c||c|c|c|c|c|c|c|}
	\hline 
	n & 0 & 1 & 2 & 3 & 4 & 5 & 6\\ 
	 \hline \hline
	 r_n(f_{2,-3}) & -2.05670 & -2 & -2.05670 & - & - & - & -\\
	 \hline
	 r_n(f_{3,-3}) & -3.25653 & -1.5 & 0 & 1.5 & 3.25652 & - &  -\\
	\hline 
	r_n(f_{4,-3}) & -6.76949 & -2.22222 & 0 & 0.66666 & 0 & -2.22222 & -6.76949\\
	\hline 
	r_n(f_{5,-3}) &  -15.65457 & -4.08333 &0 & 0.66666 & 0 & -0.66666 & 0 \\
	\hline
	r_n(f_{6,-3}) &  -38.31573 & -8.36729 &0 & 0.89204 & 0 & -0.48637 & 0\\
	\hline
	r_n(f_{7,-3}) &  -97.17273 & -18.33333 &0 & 1.4 & 0 & -0.5 & 0\\
	\hline
	\end{array}
\end{align*}
We notice that the even periods of $f_{k,-3}(z)$ for $0 < n < 2k-2$ seem to vanish and the odd periods for $k \neq 6$ appear to be rational numbers, but the outer periods $r_{0}(f_{k,-3})$ and $r_{2k-2}(f_{k,-3})$ and the odd periods of $f_{6,-3}(z)$ do not seem to be particularly nice. In order to state the result explaining these observations, following \cite{kohnenzagierrationalperiods} we introduce the functions
\[
f_{k,P}^{+}(z) := \frac{1}{2}\left(f_{k,P}(z)+f_{k,P'}(z)\right), \qquad f_{k,P}^{-}(z):= \frac{i}{2}\left(f_{k,P}(z)-f_{k,P'}(z) \right),
\]
where we put $P' := [a,-b,c]$ for $P = [a,b,c]$.

\begin{theorem}\label{theorem rationality}
\begin{enumerate}
	\item For $0<n<2k-2$ we have
	\[
	r_n\left(f_{k,P}^{\varepsilon}\right)=0,
	\]
	where $\varepsilon$ is the sign of $(-1)^{n}$.
	\item The outer periods of $f_{k,P}^{+}(z)$ are given by
	\[
	r_0(f_{k,P}^{+}) = (-1)^{k}r_{2k-2}(f_{k,P}^{+}) = -\frac{|d|^{k-\frac12}\zeta_P(k)}{2(2k-1)|\overline{\Gamma}_P|\zeta(2k)},
	\]
	where $\zeta_{P}(s) := \sum_{(x,y) \in \Z^2 \setminus \{(0,0)\}}P(x,y)^{-s}$ is the Epstein zeta function associated to $P$, and $\overline{\Gamma}_P$ denotes the stabilizer of $P$ in $\overline{\Gamma} := \Gamma/\{\pm 1\}$.
	\item Let $a_{n} \in \Q$ for odd $0 < n < 2k-2$ be coefficients such that 
	\[
	\sum_{\substack{0 < n < 2k-2 \\ n \text{ odd}}}a_{n}R_{n}(\tau) = 0
	\]
	in $S_{2k}$. Then the linear combination
	\[
	\sum_{\substack{0 < n < 2k-2 \\ n \text{ odd}}}a_{n}r_{n}(f_{k,P}) 
	\]
	of odd periods of $f_{k,P}(z)$ is rational. Similarly, if a rational linear combination of the cusp forms $R_n$ with even $0 < n < 2k-2$ vanishes in $S_{2k}$, then the corresponding linear combination of even periods of $f_{k,P}(z)$ is in $i\Q$.
\end{enumerate}
\end{theorem}

The proof of Theorem~\ref{theorem rationality} will be given in Section~\ref{section proof rationality}. A nice feature of the proof is that it also yields an exact rational formula for the linear combinations of periods of $f_{k,P}(z)$ considered in item (3) of Theorem~\ref{theorem rationality}, compare Theorem~\ref{theorem rn formula} below. In contrast, the algebraic nature of the special values $\zeta_P(k)$ appearing in the outer periods of $f_{k,P}^{+}(z)$ is more mysterious. For example, for $k = 3$ and $P = [1,1,1]$ we have $\zeta_P(3) = 6\zeta(3)L(\chi_{-3},3) = \frac{8\pi^3}{27\sqrt{3}}\zeta(3)$ (compare \cite{zagierzetafunctions}, Proposition 3), which is expected to be transcendental.

We will now give an example to illustrate item (3) of Theorem~\ref{theorem rationality}.

\begin{example}
	We first notice that for $k = 2,3,4,5,$ and $7$, the space of cusp forms of weight $2k$ is trivial. Hence, Theorem~\ref{theorem rationality} shows that in these cases the odd periods of $f_{k,P}(z)$ are rational and the even periods for $0 < n < 2k-2$ are in $i\Q$ for all positive definite forms $P$, in accordance with the numerical values given in the above table. More generally, Cohen \cite{cohen} has found relations between the cusp forms $R_{n}(\tau)$ in any weight, for example
	\[
	\sum_{\substack{0 < n \leq j-1 \\ n \text{ odd}}}\binom{j}{n}(-1)^{\frac{n-1}{2}}R_{n}(\tau)+\sum_{\substack{j \leq n < 2k-2 \\ n \text{ odd}}}\binom{2k-2-j}{n-j}(-1)^{\frac{n-1}{2}}R_{n}(\tau) = 0
	\]
	for all $0 \leq j \leq 2k-2$. Relations of this kind can be proved using the Eichler-Shimura isomorphism. From Theorem~\ref{theorem rationality} we obtain that the linear combinations
	\[
	\sum_{\substack{0 < n \leq j-1 \\ n \text{ odd}}}\binom{j}{n}(-1)^{\frac{n-1}{2}}r_{n}(f_{k,P})+\sum_{\substack{j \leq n < 2k-2 \\ n \text{ odd}}}\binom{2k-2-j}{n-j}(-1)^{\frac{n-1}{2}}r_{n}(f_{k,P})
	\]
	of odd periods of $f_{k,P}(z)$ are rational for all $0 \leq j \leq 2k-2$ and all $P$. For instance, for $k = 6, d = -3$, and $j = 4$, we find that the linear combination
	\[
	10r_{1}(f_{6,-3}) - 24r_{3}(f_{6,-3}) + 6r_{5}(f_{6,-3})
	\]
	is rational. Indeed, plugging in the numerical values given in the above table, we find that this linear combination of periods is numerically close to the integer $-108$. Using our exact formula for the periods of $f_{k,P}(z)$ given in Theorem~\ref{theorem rn formula} below one can show that this really is the correct value.
\end{example}

We will explain the results from Theorem~\ref{theorem rationality} by using the following connection between the periods of $f_{k,P}(z)$ and special values of (derivatives of) the locally harmonic Maass forms $\mathcal{H}_{1-k,n}(\tau)$.

\begin{proposition}\label{proposition special value}
	We have
	\[
	r_{n}(f_{k,P}) = \frac{(-1)^{k-1}|d|^\frac{k-1}{2}}{2^{k-2}(k-1)!|\overline{\Gamma}_P|}R_{2-2k}^{k-1}\mathcal{H}_{1-k,n}(\tau_{P}),
	\]
	where $\tau_{P} \in \H$ denotes the CM point characterized by $P(\tau_{P},1) = 0$, and $R_{2-2k}^{k-1} := R_{-2}\circ \dots \circ R_{2-2k}$ is an iterated version of the Maass raising operator $R_{\kappa} := 2i\frac{\partial}{\partial \tau}+\kappa v^{-1}$.
\end{proposition}

The proof of the above proposition can be found in Section~\ref{section proof rationality}. In Theorem~\ref{theorem rn formula} below we will apply the iterated raising operator $R_{2-2k}^{k-1}$ to the splitting of $\mathcal{H}_{1-k,n}(\tau)$ from Theorem~\ref{theorem splitting}. Together with Proposition~\ref{proposition special value} we obtain an explicit formula for the periods $r_{n}(f_{k,P})$, which we will then use to prove Theorem~\ref{theorem rationality}.

Eventually, we remark that the methods of this work can be used to study the rationality of periods of certain linear combinations of the meromorphic modular forms $f_{k,P}(z)$, similar to \cite{ls}. For example, in analogy to Theorem~2.4 from \cite{ls}, one can show that certain linear combinations of Hecke-translates of $f_{k,P}(z)$ have rational periods.

We start with a section on the necessary preliminaries. In the remaining sections, we give the proofs of the above results.

\section{Preliminaries}\label{section preliminaries}

\subsection{Derivatives of Eichler integrals}

The holomorphic and the non-holomorphic Eicher integrals defined in \eqref{eq eichler integrals} are related by the iterated raising operator $R_{2-2k}^{k-1} = R_{-2}\circ \dots \circ R_{2-2k}$, with $R_{\kappa} = 2i\frac{\partial}{\partial \tau} + \kappa v^{-1}$, as follows.

\begin{proposition}\label{proposition eichler integral relations}
		For any holomorphic modular form $f \in M_{2k}$ we have the relation
		\[
		R_{2-2k}^{k-1}f^*(\tau) = -\frac{(2k-2)!}{(4\pi)^{2k-1}}\overline{R_{2-2k}^{k-1}\mathcal{E}_f(\tau)}.
		\]
\end{proposition}

\begin{proof}
	Using the Fourier expansions of the Eichler integrals given in \eqref{eq eichler integrals} we see that the claim is equivalent to
\begin{align}\label{eq eichler integrals simplified}
R_{2-2k}^{k-1}\left(\Gamma(2k-1,4\pi n v)e(-n\tau)\right) = (2k-2)! \overline{R_{2-2k}^{k-1}e(n\tau)}
\end{align}
for all $n \geq 1$. Following \cite{bruinierhabil}, Section 1.3, we consider the function
\[
\mathcal{W}_{\kappa,s}(y) := |y|^{-\kappa/2}W_{\frac{\kappa}{2}\sgn(y),s-\frac{1}{2}}(|y|), \qquad (\kappa \in \R, \ s \in \C, \ y \in \R \setminus \{0\}),
\]
where $W_{\nu,\mu}(y)$ is the usual $W$-Whittaker function. At $s = 1-\frac{\kappa}{2}$ it simplifies to
\[
\mathcal{W}_{\kappa,1-\frac{\kappa}{2}}(y) = \begin{cases}
e^{-y/2}, & \text{if $y > 0$}, \\
e^{-y/2}\Gamma(1-\kappa,|y|), & \text{if $y < 0$}.
\end{cases}
\]
In particular, \eqref{eq eichler integrals simplified} is equivalent to
\begin{align}\label{eq eichler integrals simplified 2}
R_{2-2k}^{k-1}\left(\mathcal{W}_{2-2k,k}(-4\pi n v)e(-nu) \right) = (2k-2)! \overline{R_{2-2k}^{k-1}\left(\mathcal{W}_{2-2k,k}(4\pi n v)e(nu)\right)}.
\end{align}
On the other hand, using (13.4.33) and (13.4.31) in \cite{abramowitz}, we obtain the formula
\begin{align*}
&R_{\kappa}\left(\mathcal{W}_{\kappa,s}(4\pi m v)e(mu)\right) \\
&\quad = \begin{cases}
-4\pi |m|\left(s+\frac{\kappa}{2}\right)\left(s-\frac{\kappa}{2}-1\right)\mathcal{W}_{\kappa+2,s}(4\pi m v)e(mu), & \text{if $m < 0$}, \\
-4\pi |m|\mathcal{W}_{\kappa+2,s}(4\pi m v)e(mu), & \text{if $m > 0$},
\end{cases}
\end{align*}
for $u+iv \in \H$ and $m \in \R \setminus \{0\}$. This easily implies \eqref{eq eichler integrals simplified 2} and finishes the proof.
\end{proof}

\subsection{Periods of cusp forms and non-cusp forms}\label{section prelims periods}

We recall from \cite{kohnenzagierrationalperiods} that the period polynomial of a cusp form $f \in S_{2k}$ is defined by
\[
r_f(\tau) := \int_0^{i\infty} f(z)(z - \tau)^{2k-2}dz = \sum_{n = 0}^{2k-2}i^{-n+1}\binom{2k-2}{n}r_n(f)\tau^{2k-2-n}.
\]
The even period polynomial $r_f^+(\tau)$ of $f$ is defined as $-i$ times the even part of $r_f(\tau)$, and the odd period polynomial $r_f^-(\tau)$ of $f$ is the odd part of $r_f(\tau)$, such that $r_f(\tau) = r_f^-(\tau)+ir_f^+(\tau)$. It is well-known that the errors of modularity of the Eichler integrals of $f$ can be expressed in terms of the period function of $f$ as
\begin{align}\label{holomorphiceichlertrafo}
	\mathcal{E}_{f}(\tau)\Bigl|_{2-2k}(I-S) &= \frac{(-2\pi i )^{2k-1}}{(2k-2)!}r_{f}(\tau)
\end{align}
and
\begin{align}\label{nonholomorphiceichlertrafo}
	f^*(\tau)\Bigl|_{2-2k}(I-S) &= (-2i)^{1-2k}r_{f}^c(\tau),
\end{align}
	where $r_f^c(\tau) := \overline{r_f(\overline{\tau})}$ is the polynomial whose coefficients are the complex conjugates of the coefficients of $r_f(\tau)$.

For $0 \leq n \leq 2k-2$ the periods of a (not necessarily cuspidal) modular form $f \in M_{2k}$ are defined by 
\[
r_n(f) := \frac{n!}{(2\pi)^{n+1}}L_f(n+1),
\]
where $L_f(s)$ denotes the usual $L$-function associated to $f$. For a cusp form $f \in S_{2k}$ this agrees with the definition~\eqref{definition periods}. Note that the functional equation of the $L$-function of $f$ implies the symmetry $r_n(f) = (-1)^k r_{2k-2-n}(f)$. 

The periods of the normalized Eisenstein series $E_{2k}$ are given by
\begin{align}\label{periodEisenstein}
r_{n}(E_{2k}) = \begin{dcases}
-\frac{\pi \zeta(2k-1) }{(2k-1)\zeta(2k)}, & \text{if $n=0$,}\\
0, & \text{if $0<n<2k-2$ is even,}\\
(-1)^{\frac{n-1}{2}}\frac{2k B_{n+1}B_{2k-1-n}}{B_{2k}(n+1)(2k-1-n)}, & \text{if $0<n<2k-2$ is odd,}\\
(-1)^{k-1}\frac{\pi \zeta(2k-1) }{(2k-1)\zeta(2k)}, & \text{if $n =2k-2$,}
\end{dcases}
\end{align}
 compare p.240 of \cite{kohnenzagierrationalperiods}. Following \cite{zagierperiods}, we define the corresponding period function by
\begin{equation}\label{EisenSplit}
r_{E_{2k}}(\tau) :=\frac{1}{2k-1}\left(\tau^{2k-1}+\frac{1}{\tau}\right) + \sum_{n=0}^{2k-2}i^{1-n}\binom{2k-2}{n}r_{n}(E_{2k})\tau^{2k-2-n} = r^-_{E_{2k}}(\tau)+ir^+_{E_{2k}}(\tau)
\end{equation}
with 
\begin{align}\label{EisenEven}
r^+_{E_{2k}}(\tau) := r_{0}(E_{2k})(\tau^{2k-2}-1)
\end{align}
and 
\begin{align}\label{EisenPeriodOddPoly}
r^-_{E_{2k}}(\tau)
:=\frac{2k(2k-2)!}{B_{2k}}\sum_{\substack{-1\leq n\leq 2k-1, \\ \text{$n$ odd}}}\frac{B_{n+1}B_{2k-1-n}}{(n+1)!(2k-1-n)!}\tau^{2k-2-n}.
\end{align}
Then we also have 
\begin{align}\label{EisenEichlerTrafo}
\mathcal{E}_{E_{2k}}(\tau)\Bigl|_{2-2k}(I-S) = \frac{(-2\pi i )^{2k-1}}{(2k-2)!}r_{E_{2k}}(\tau).
 \end{align}

\subsection{The cusp forms $R_{n}(\tau)$}\label{section period polynomial Rn}

Recall that $R_n(\tau)$ denotes the unique cusp form of weight $2k$ for $\Gamma$ which satisfies the inner product formula $\langle f,R_n \rangle = r_n(f)$ for every cusp form $f(\tau) \in S_{2k}$. We will need the following well-known series representation of $R_n(\tau)$.

	\begin{proposition}\label{proposition Rn}
		For $0 < n < 2k-2$ the cusp form $R_n(\tau) \in S_{2k}$ is given by
		\[
R_{n}(\tau) = \frac{2^{2k-3}}{i^{2k-1-n}\binom{2k-2}{n}\pi}\sum_{M \in \Gamma}\tau^{-n-1}\Bigl|_{2k}M.
		\]
		Moreover, for $n = 0$ or $n = 2k-2$ it can be constructed as
		\[
		R_{2k-2}(\tau) = (-1)^k R_{0}(\tau) = 2^{2k-1}\sum_{m=1}^\infty P_{2k,m}(\tau),
		\]
		where $P_{2k,m}(\tau) := \sum_{M \in \Gamma_\infty \backslash \Gamma}e(m\tau)|_{2k}M$ with $\Gamma_\infty := \{\pm\left(\begin{smallmatrix}1 & n \\ 0 & 1 \end{smallmatrix} \right): n\in \Z\}$ is the usual cuspidal Poincar\'e series of weight $2k$.
	\end{proposition}

	\begin{proof}
		The above series representation of $R_n(\tau)$ for $0 < n < 2k-2$ was first given by Cohen in \cite{cohen}. We also refer the reader to the Lemma in Section 1.2 of \cite{kohnenzagierrationalperiods} for a proof.
		
		For $n = 2k-2$ and $f(\tau) = \sum_{m=1}^\infty c_f(m)e(m\tau) \in S_{2k}$ we have
		\[
		\left\langle f,2^{2k-1}\sum_{m=1}^\infty  P_{2k,m}\right\rangle = \frac{(2k-2)!}{(2\pi)^{2k-1}}\sum_{m=1}^{\infty}\frac{c_f(m)}{m^{2k-1}} = \frac{(2k-2)!}{(2\pi)^{2k-1}}L_f(2k-1) = r_{2k-2}(f) = \langle f,R_{2k-2} \rangle
		\]
		by the usual Petersson inner product formula for $P_{2k,m}(\tau)$. In other words, this means that $2^{2k-1}\sum_{m=1}^\infty P_{2k,m}(\tau)$ converges weakly to $R_{2k-2}(\tau)$, which implies $2^{2k-1}\sum_{m=1}^\infty P_{2k,m}(\tau) = R_{2k-2}(\tau)$ since $S_{2k}$ is finite-dimensional.
	\end{proof}

We state the explicit formulas of Kohnen and Zagier \cite{kohnenzagierrationalperiods} for the even and odd period polynomials of the cusp forms $R_n(\tau)$, in a form that is convenient for our purposes. 

\begin{theorem}[\cite{kohnenzagierrationalperiods}, Theorem 1']\label{Periods R_n}
For $0\leq n\leq 2k-2$ even, the odd period polynomial of $R_n(\tau)$ is given by
\begin{multline*}
\left(\frac{i}{2}\right)^{2k-2}ir^-_{R_n}(\tau)
=i^{n+1}\left(\frac{B_{n+1}(\tau)}{n+1}- \frac{B_{2k-1-n}(\tau)}{2k-1-n}\right)\Bigl|_{2-2k}(I-S)\\
+i^{n+1}(\tau^n-\tau^{2k-2-n})+i(\delta_{n=0}+(-1)^k\delta_{n=2k-2})r^-_{E_{2k}}(\tau),
\end{multline*}
and for $0< n< 2k-2$ odd, the even period polynomial of $R_n(\tau)$ is given by
\begin{multline*}
\left(\frac{i}{2}\right)^{2k-2}r^+_{R_n}(\tau)
=-i^{n+1}\left(\frac{B_{n+1}(\tau)}{n+1}+ \frac{B_{2k-1-n}(\tau)}{2k-1-n}\right)\Bigl|_{2-2k}(I-S)  \\
-i^{n+1}(\tau^{n} + \tau^{2k-2-n})+r_n(E_{2k})(\tau^{2k-2}-1).
\end{multline*}
\end{theorem}

\section{The proof of Proposition~\ref{proposition diffops}}\label{section proof proposition diffops}

First note that the Poincar\'e series defined in \eqref{eq Peterssons Poincare series} satisfy the differential equations
\begin{align*}
R_{-2\ell,\tau}H_{k,\ell}(z,\tau) &=  (k-\ell)H_{k,\ell-1}(z,\tau), \\
L_{-2\ell,\tau}H_{k,\ell}(z,\tau) &= (k+\ell)H_{k,\ell+1}(z,\tau),
\end{align*}
which can be checked by a direct computation. Furthermore, by Bol's identity the iterated derivative $\mathcal{D}^{2k-1}$ can be expressed in terms of the iterated raising operator by
\[
\mathcal{D}^{2k-1} = -\frac{1}{(4\pi)^{2k-1}}R_{2-2k}^{2k-1},
\]
compare equation (56) in Zagier's part of \cite{zagier123}. In particular, we find
\begin{align*}
\xi_{2-2k,\tau}H_{k,k-1}(z,\tau) &= v^{-2k}\overline{L_{2-2k,\tau}H_{k,k-1}(z,\tau)} = (2k-1)v^{-2k}\overline{H_{k,k}(z,\tau)}, \\
\mathcal{D}_\tau^{2k-1}H_{k,k-1}(z,\tau) &= -\frac{1}{(4\pi)^{2k-1}}R_{2-2k, \tau}^{2k-1}H_{k,k-1}(z,\tau)=-\frac{(2k-1)!}{(4\pi)^{2k-1}}H_{k,-k}(z,\tau).
\end{align*}
For $\tau \in \C\setminus i\R_{\geq 0}$ we will need the evaluation 
\begin{align}\label{yintegraleval}
\int_{0}^{\infty} \frac{y^ndy}{(iy-\tau)^{2k}} = \frac{i^{n+1}}{2k-1}\binom{2k-2}{n}^{-1}\tau^{n+1-2k},
\end{align}
which can be shown directly for $\tau = iv$ with $v < 0$ and then follows by analytic continuation for all $\tau \in \C \setminus i\R_{\geq 0}$.
For $0 < n < 2k-2$ we can now compute
\begin{align*}
\xi_{2-2k}\mathcal{H}_{1-k,n}(\tau) &=(2k-1) \frac{(2i)^{2k-2}}{2\pi}v^{-2k}\overline{\sum_{M\in\Gamma}\left(v^{2k}\int_{0}^{\infty} \frac{y^ndy}{(iy-\overline{\tau})^{2k}}\right)\Bigl|_{-2k,\tau}M}\\
&=(-i)^{n+1}\frac{(2i)^{2k-2}}{2\pi}\binom{2k-2}{n}^{-1}v^{-2k}\overline{\sum_{M\in\Gamma}v^{2k}\overline{\tau}^{n+1-2k}\Bigl|_{-2k,\tau}M} \\
&=(-i)^{n+1}\frac{(2i)^{2k-2}}{2\pi}\binom{2k-2}{n}^{-1}\sum_{M\in\Gamma}\tau^{n+1-2k}\Bigl|_{2k,\tau}M\\
&= (-1)^{k-1}R_{2k-2-n}(\tau)= - R_n(\tau),
\end{align*}
and 
\begin{align*}
\mathcal{D}^{2k-1}\mathcal{H}_{1-k,n}(\tau) &=(-1)^k\frac{(2k-1)!}{2(2\pi)^{2k}}\sum_{M\in\Gamma}\left(\int_{0}^{\infty}  \frac{y^ndy}{(iy-\tau)^{2k}}\right)\Bigl|_{2k,\tau}M\\
&=i^{n+1}(-1)^k\frac{(2k-2)!}{2(2\pi)^{2k}}\binom{2k-2}{n}^{-1}\sum_{M\in\Gamma}\tau^{n+1-2k}\Bigl|_{2k,\tau}M\\
& = (-1)^{n+1} \frac{(2k-2)!}{(4\pi)^{2k-1}}(-1)^k R_{2k-2-n}(\tau) = (-1)^{n+1}\frac{(2k-2)!}{(4\pi)^{2k-1}}R_{n}(\tau),
\end{align*}
where we used the series representation of $R_n(\tau)$ from Proposition~\ref{proposition Rn}.

However, for $n = 0$ or $n = 2k-2$ the series representation of $R_n(\tau)$ used above does not converge, so we need to proceed differently. By the symmetries $\mathcal{H}_{1-k,n}(\tau) = (-1)^k \mathcal{H}_{1-k,2k-2-n}(\tau)$ and $R_{n}(\tau) = (-1)^k R_{2k-2-n}(\tau)$ it suffices to treat the case $n = 2k-2$. First, we write
\begin{align*}
\xi_{2-2k}\mathcal{H}_{1-k,2k-2}(\tau) &= 2(2k-1) \frac{(2i)^{2k-2}}{2\pi}v^{-2k}\overline{\sum_{M\in\Gamma_\infty\backslash\Gamma}\left(v^{2k}\int_{0}^{\infty} \sum_{m \in \Z} \frac{y^{2k-2}dy}{(iy-\overline{\tau}+m)^{2k}}\right)\Bigl|_{-2k,\tau}M}.
\end{align*}
Using the Lipschitz formula
\begin{align}\label{eq Lipschitz formula}
\sum_{m \in \Z}(z + m)^{-2k} = \frac{(2\pi i )^{2k}}{(2k-1)!}\sum_{m = 1}^\infty m^{2k-1}e(mz),
\end{align}
which is valid for $z \in \H$, we can rewrite the integral as
\begin{align*}
\int_{0}^{\infty} \sum_{m \in \Z} \frac{y^{2k-2}dy}{(iy-\overline{\tau}+m)^{2k}} &= \frac{(2\pi i)^{2k}}{(2k-1)!}\int_{0}^{\infty}y^{2k-2} \sum_{m = 1}^{\infty}m^{2k-1}e^{-2\pi m y}e(-m\overline{\tau})dy \\
&= \frac{2\pi (-1)^k}{2k-1}\sum_{m=1}^\infty e(-m \overline{\tau}).
\end{align*}
Putting everything together, we obtain
\begin{align*}
\xi_{2-2k}\mathcal{H}_{1-k,2k-2}(\tau) &= -2^{2k-1}v^{-2k}\overline{\sum_{M \in \Gamma_\infty\backslash \Gamma}\left(v^{2k}\sum_{m=1}^\infty e(-m \overline{\tau})\right)\Bigl|_{-2k,\tau}M} \\
&= -2^{2k-1}\sum_{m=1}^\infty \sum_{M \in \Gamma_\infty \backslash \Gamma}e(m \tau)\Bigl|_{2k,\tau}M \\
&= -2^{2k-1}\sum_{m=1}^\infty P_{2k,m}(\tau) = - R_{2k-2}(\tau),
\end{align*}
where we used Proposition~\ref{proposition Rn} in the last equality.

To compute $\mathcal{D}^{2k-1}\mathcal{H}_{1-k,2k-2}(\tau)$, we write as before
\begin{align*}
\mathcal{D}^{2k-1}\mathcal{H}_{1-k,2k-2}(\tau) &= -\frac{(2i)^{2k-2}(2k-1)!}{\pi (4\pi)^{2k-1}}\sum_{M \in \Gamma_\infty\backslash\Gamma}\left(\int_0^\infty \sum_{m \in \Z}\frac{y^{2k-2}}{(iy-\tau+m)^{2k}} dy\right)\Bigl|_{2k,\tau}M.
\end{align*}
It now suffices to prove the identity
\begin{align}\label{eq D identity}
\int_0^\infty \sum_{m \in \Z}\frac{y^{2k-2}}{(iy-\tau+m)^{2k}}dy = \frac{2^{2k-1}\pi}{(2i)^{2k-2}(2k-1)}\cdot\frac{1}{1-e(\tau)}
\end{align}
for $\tau \in \H$ with $0 < u < 1$, since we can write $\frac{1}{1-e(\tau)} = 1+\sum_{m=1}^\infty e(m\tau)$ and then finish the computation of $\mathcal{D}^{2k-1}\mathcal{H}_{1-k,2k-2}(\tau)$ in the same way as for the $\xi_{2-2k}$-image, using Proposition~\ref{proposition Rn}. Note that we cannot directly apply the Lipschitz formula \eqref{eq Lipschitz formula} to prove the identity \eqref{eq D identity} since $iy-\tau$ has negative imaginary part for $0 < y < v$. However, since both sides of \eqref{eq D identity} are holomorphic functions on the vertical strip consisting of all $\tau \in \C$ with $0 < u < 1$, it is enough to prove \eqref{eq D identity} for all $\tau$ in this strip with $v < 0$ and then use analytic continuation. Under the assumption $v < 0$ we can apply the Lipschitz formula~\eqref{eq Lipschitz formula} and obtain similarly as before
\begin{align*}
	\int_0^\infty \sum_{m \in \Z}\frac{y^{2k-2}}{(iy-\tau+m)^{2k}}dy
	&= \frac{(2\pi i)^{2k}}{(2k-1)!} \int_0^\infty y^{2k-2}\sum_{m=1}^\infty m^{2k-2}e^{-2\pi m y}e(-m\tau)dy \\ 
	&= \frac{(2\pi i)^{2k}}{(2k-1)!}\cdot \frac{(2k-2)!}{(2\pi)^{2k-1}}\sum_{m=1}^\infty e(-m\tau) \\
	&= -\frac{2^{2k-1}\pi}{(2i)^{2k-2}(2k-1)}\left(\frac{1}{1-e(-\tau)}-1 \right) \\
	&= \frac{2^{2k-1} \pi}{(2i)^{2k-2}(2k-1)}\cdot\frac{1}{1-e(\tau)},
\end{align*}
	which yields \eqref{eq D identity} and concludes the computation of $\mathcal{D}^{2k-1}\mathcal{H}_{1-k,2k-2}(\tau)$. This finishes the proof of Proposition~\ref{proposition diffops}.

\section{The proof of Theorem~\ref{theorem splitting}}\label{section proof theorem splitting}

	In this section we prove Theorem~\ref{theorem splitting}, i.e., the decomposition of the locally harmonic Maass form $\mathcal{H}_{1-k,n}(\tau)$ into a sum of a locally polynomial part and Eichler integrals of the cusp forms $R_n(\tau)$. For brevity, we let
 \begin{align*}
 \widetilde{\mathcal{H}}_{1-k,n}(\tau) &:= \mathcal{P}_{1-k,n}(\tau) + (-1)^{n+1}\frac{(2k-2)!}{(4\pi )^{2k-1}}\mathcal{E}_{R_n}(\tau) - R_{n}^*(\tau) \\
 &\quad-((-1)^k\delta_{n=0}+\delta_{n=2k-2})\frac{(2k-2)!}{(2\pi)^{2k-1}}\mathcal{E}_{E_{2k}}(\tau)
 \end{align*}
 be the expression on the right-hand side of Theorem~\ref{theorem splitting}. Then the theorem is equivalent to the identity $\mathcal{H}_{1-k,n}(\tau) = \widetilde{\mathcal{H}}_{1-k,n}(\tau)$. In order to prove this identity, we show that $\widetilde{\mathcal{H}}_{1-k,n}(\tau)$ is a locally harmonic Maass form of weight $2-2k$ with the same singularities as $\mathcal{H}_{1-k,n}(\tau)$ and the same images under the differential operators $\xi_{2-2k}$ and $\mathcal{D}^{1-2k}$. This implies that the difference $\mathcal{H}_{1-k,n}(\tau)-\widetilde{\mathcal{H}}_{1-k,n}(\tau)$ is a polynomial on $\H$ which is also modular of negative weight $2-2k$, and hence vanishes identically.
 
 We first show that $\widetilde{\mathcal{H}}_{1-k,n}(\tau)$ is modular. To this end, we study the modularity properties of the local polynomial $\mathcal{P}_{1-k,n}(\tau)$. We start by rewriting $\mathcal{P}_{1-k,n}(\tau)$ in a more convenient form. 

\begin{lemma}\label{lemma alternative polynomial}
For $\tau \in \H$ we have 
\begin{align*}
\mathcal{P}_{1-k,n}(\tau) &= r_n(E_{2k}) +\frac{(-i)^{n+1}}{n+1}B_{n+1}(\tau)+\frac{i^{n+1}}{2k-1-n}B_{2k-1-n}(\tau)
 \\
 &\quad +\frac{(-i)^{n+1}}{2}\left(\tau^n-(-1)^n\tau^{2k-2-n}\right)+\frac{i^{1-n}}{4}\sum_{M\in\Gamma}((\sgn(u)-\sgn(M))\tau^n)\Bigl |_{2-2k}M,
\end{align*}
where
\[
\sgn\left(\begin{pmatrix} a & b\\ c& d \end{pmatrix}\right):=\begin{cases}
\sgn(ac), & \text{if $ac\neq 0$,}\\
\sgn(bd), & \text{otherwise.}
\end{cases}
\]
\end{lemma}

\begin{proof}
For simplicity, we sketch the computations in the case that $\tau\notin\bigcup_{M \in \Gamma}M(i\R_{+})$ and leave the general case to the reader. We first rewrite
\begin{multline*}
\sum_{\substack{M = \left(\begin{smallmatrix}a & b \\ c & d \end{smallmatrix}\right)\in\Gamma \\ ac>0, \Re(M\tau)> 0}}(\tau^n -(-1)^n \tau^{2k-2-n})\Bigl |_{2-2k}M = \sum_{\substack{M = \left(\begin{smallmatrix}a & b \\ c & d \end{smallmatrix}\right)\in\Gamma \\ ac>0, \Re(M\tau)> 0}}\tau^n \Bigl |_{2-2k}M - \tau^n|_{2-2k}SM \\  
=\sum_{\substack{M = \left(\begin{smallmatrix}a & b \\ c & d \end{smallmatrix}\right)\in\Gamma \\ ac \Re(M\tau) < 0 }}\sgn(M)\tau^n \Bigl |_{2-2k}M 
=-\frac12\sum_{\substack{M = \left(\begin{smallmatrix}a & b \\ c & d \end{smallmatrix}\right)\in\Gamma \\ ac \neq 0 }}\left((\sgn(u)-\sgn(M))\tau^n\right) \Bigl |_{2-2k}M. 
\end{multline*}
Here we used that $ac>0$ if and only if $SM = \left(\begin{smallmatrix}\alpha & \beta \\ \gamma & \delta \end{smallmatrix}\right)$ has $\alpha\gamma <0$ and $\sgn(\Re(SM\tau))\neq \sgn(\Re(M\tau))$.

Using $\frac{B_m(\tau+1)}{m} = \frac{B_m(\tau)}{m} + \tau^{m-1}$ we write
\begin{align*}
\frac{\mathbb{B}_m(\tau)}{m} -\frac{B_m(\tau)}{m} &= 
\begin{cases}-\sum_{j=1}^{\lfloor u \rfloor}(\tau-j)^{m-1},& \text{if $u> 0$,} \\
\sum_{j=0}^{\lfloor -u\rfloor } (\tau+j)^{m-1}, & \text{if $u<0$,}
\end{cases} \\
&= \frac12\tau^{m-1}+\frac12\sum_{j\in\Z}(\sgn(j)-\sgn(u+j))(\tau+j)^{m-1} \\
&=\frac12\tau^{m-1}+\frac14\sum_{\substack{M = \left(\begin{smallmatrix}a & b \\ c & d \end{smallmatrix}\right)\in\Gamma \\ c = 0 }}(\sgn(M)-\sgn(u))\tau^{m-1}\Bigl |_{2-2k}M.
\end{align*}
Hence, we find after a short computation
\begin{align*}
\Biggl(\frac{\mathbb{B}_{n+1}(\tau)}{n+1}&-(-1)^n\frac{\mathbb{B}_{2k-1-n}(\tau)}{2k-1-n}\Biggr)-\Biggl(\frac{B_{n+1}(\tau)}{n+1}-(-1)^n\frac{B_{2k-1-n}(\tau)}{2k-1-n}\Biggr) \\
&=\frac12\left(\tau^n -(-1)^n\tau^{2k-2-n}\right)- \frac14\sum_{\substack{M = \left(\begin{smallmatrix}a & b \\ c & d \end{smallmatrix}\right)\in\Gamma \\ ac = 0 }}(\sgn(u)-\sgn(M))\tau^{n}\Bigl |_{2-2k}M.
\end{align*}
Gathering everything together, we obtain the formula given in the lemma.
\end{proof}

We can now state the transformation law for the locally polynomial part $\mathcal{P}_{1-k,n}(\tau)$.

\begin{lemma}\label{lemma trafo P}
For $\tau \in \H$ we have
\begin{align*}
\mathcal{P}_{1-k,n}(\tau)\Bigl|_{2-2k}(I-T) = 0
\end{align*}
and
\begin{align*}
\mathcal{P}_{1-k,n}(\tau)\Bigl|_{2-2k}(I-S) = \begin{dcases}
-\left(\frac{i}{2}\right)^{2k-2}r^+_{R_n}(\tau),&\text{if $n$ is odd,} \\
\begin{aligned} &-\left(\frac{i}{2}\right)^{2k-2}ir^-_{R_n}(\tau)  \\
&\quad+i(\delta_{n=0} +(-1)^k \delta_{n=2k-2})r_{E_{2k}}(\tau),
\end{aligned}&\text{if $n$ is even.}
\end{dcases}
\end{align*}
\end{lemma}

\begin{proof} 
For the $T$-transformation, we write
$$
\left(\sum_{M\in\Gamma}\left((\sgn(u)-\sgn(M))\tau^n\right)\Bigl |_{2-2k}M\right)\Bigl|_{2-2k}T =
\sum_{M\in\Gamma}\left((\sgn(u)-\sgn(MT^{-1}))\tau^n\right)\Bigl |_{2-2k}M.
$$
Now one can show that the only matrices with $\sgn(MT^{-1})\neq \sgn(M)$ are $M\in\pm \{I, S, T, ST\}$. Therefore we get
\begin{align*}
\Biggl(&\sum_{M\in\Gamma}((\sgn(u)-\sgn(M))\tau^n)\Bigl |_{2-2k}M\Biggr)\Bigl|_{2-2k}(I-T)\\
&= 2\sum_{M\in\{I, S, T, ST\}}((\sgn(MT^{-1})-\sgn(M))\tau^n)\Bigl |_{2-2k}M\\
&= 2(-\tau^n + (-1)^n\tau^{2k-2-n} - (\tau+1)^n + (-1)^n(\tau+1)^{2k-2-n}).
\end{align*}
The Bernoulli polynomials satisfy $
\frac{B_m(\tau+1)}{m} = \frac{B_m(\tau)}{m} + \tau^{m-1}$,
so all in all we obtain
\begin{align*}
&\mathcal{P}_{1-k,n}(\tau)\Bigl|_{2-2k}(I-T) = \frac12 (-i)^{n+1}\Bigl(-2\tau^n +2(-1)^n\tau^{2k-2-n} + \tau^n -(\tau+1)^n -(-1)^n\tau^{2k-2-n} \\
&\qquad+ (-1)^n(\tau+1)^{2k-2-n}+\tau^n -(-1)^n\tau^{2k-2-n}+(\tau+1)^n -(-1)^n(\tau+1)^{2k-2-n} \Bigr) = 0.
\end{align*}

For the $S$-transformation, note that
\begin{align*}
\left(\sum_{M\in\Gamma}\left((\sgn(u)-\sgn(M))\tau^n\right)\Bigl |_{2-2k}M\right)\Bigl|_{2-2k}S 
&=\sum_{M\in\Gamma}\left((\sgn(u)-\sgn(MS^{-1}))\tau^n\right)\Bigl |_{2-2k}M \\
&= \sum_{M\in\Gamma}\left((\sgn(u)-\sgn(M))\tau^n\right)\Bigl |_{2-2k}M,
\end{align*}
since we have $\sgn(MS^{-1}) = \sgn(M)$ for all $M\in\Gamma$. Furthermore, if $n$ is odd, then $(-i)^{n+1} = i^{n+1}$ and Theorem \ref{Periods R_n} gives
\begin{align*}
&\Bigl(r_n(E_{2k}) +\frac{(-i)^{n+1}}{n+1}B_{n+1}(\tau)+\frac{i^{n+1}}{2k-1-n}B_{2k-1-n}(\tau) \\
&\qquad+\frac{(-i)^{n+1}}{2}(\tau^n-(-1)^n \tau^{2k-2-n}\Bigr)\Bigl|_{2-2k}(I-S) \\
&= -r_n(E_{2k})(\tau^{2k-2}-1) +i^{n+1}\left(\frac{B_{n+1}(\tau)}{n+1}+\frac{B_{2k-1-n}(\tau)}{2k-1-n}\right)\Bigl|_{2-2k}(I-S) \\
&\qquad +i^{n+1}(\tau^n+\tau^{2k-2-n}) \\
&=-\left(\frac{i}{2}\right)^{2k-2}r^+_{R_n}(\tau).
\end{align*}
If $n$ is even, then $(-i)^{n+1} = -i^{n+1}$ and Theorem~\ref{Periods R_n} yields
\begin{align*}
&\Bigl( r_n(E_{2k}) +\frac{(-i)^{n+1}}{n+1}B_{n+1}(\tau)+\frac{i^{n+1}}{2k-1-n}B_{2k-1-n}(\tau) \\
&\qquad +\frac{(-i)^{n+1}}{2}(\tau^n-(-1)^n\tau^{2k-2-n})\Bigr)\Bigl|_{2-2k}(I-S) \\
&=  -r_n(E_{2k})(\tau^{2k-2}-1) +i^{n+1}\left(-\frac{B_{n+1}(\tau)}{n+1}+\frac{B_{2k-1-n}(\tau)}{2k-1-n}\right)\Bigl|_{2-2k}(I-S) \\
&\qquad -i^{n+1}(\tau^n-\tau^{2k-2-n}) \\
&=-i\left(\frac{i}{2}\right)^{2k-2}r^-_{R_n}(\tau)+i(\delta_{n=0} +(-1)^k \delta_{n=2k-2})r_{E_{2k}}(\tau).
\end{align*}
In the last step we used that $r_n(E_{2k}) =0$ for even $n$ unless $n=0$ or $n=2k-2$, as well as \eqref{EisenSplit} and \eqref{EisenEven}. Pulling everything together, we obtain the stated result.
\end{proof}

Combining the above results, we obtain the modularity of the function $\widetilde{\mathcal{H}}_{1-k,n}(\tau)$ defined in the beginning of this section.

\begin{proposition}\label{proposition modularity Htilde}
	The function $\widetilde{\mathcal{H}}_{1-k,n}(\tau)$ transforms like a modular form of weight $2-2k$ for $\Gamma$.
\end{proposition}

\begin{proof}
	Since the Eichler integrals are one-periodic by definition and the local polynomial $\mathcal{P}_{1-k,n}(\tau)$ is one-periodic by Lemma~\ref{lemma trafo P}, the same is true for $\widetilde{\mathcal{H}}_{1-k,n}(\tau)$.

	Next, we show the invariance of $\widetilde{\mathcal{H}}_{1-k,n}(\tau)$ under $S$. Using \eqref{holomorphiceichlertrafo} and \eqref{nonholomorphiceichlertrafo}, and the fact that $R_n(\tau)$ has real Fourier coefficients, we compute
\begin{align*}
\left((-1)^{n+1}\frac{(2k-2)!}{(4\pi )^{2k-1}}\mathcal{E}_{R_n}(\tau) - R_{n}^*(\tau)\right)\Bigl|_{2-2k}(I-S)
&=\left(\frac{i}{2}\right)^{2k-2}
\begin{cases}
r_{R_n}^+(\tau), & \text{if $n$ is odd,}\\
ir_{R_n}^-(\tau),& \text{if $n$ is even.}
\end{cases}
\end{align*}
Combining this with Lemma~\ref{lemma trafo P} (and \eqref{EisenEichlerTrafo} if $n = 0$ or $n = 2k-2$), we see that 
\[
\widetilde{\mathcal{H}}_{1-k,n}(\tau)\Bigl|_{2-2k}(I-S) = 0.
\]
This finishes the proof.
\end{proof}

Next, we determine the singularities of $\mathcal{H}_{1-k,n}(\tau)$ and $\widetilde{\mathcal{H}}_{1-k,n}(\tau)$. Here, we say that a function $f$ has a \emph{singularity of type $g$} at a point $\tau_{0}$ if there exists a neighbourhood $U$ of $\tau_{0}$ on which $f-g$ is harmonic. 

\begin{lemma}\label{jumps}
The functions $\mathcal{H}_{1-k,n}(\tau)$ and $\widetilde{\mathcal{H}}_{1-k,n}(\tau)$ are harmonic on $\H\setminus \bigcup_{M \in \Gamma}M(i\R_{+})$. At a point $\tau_{0} \in \H$ they have a singularity of type
	\[
 \frac{i^{1-n}}{4}\sum_{\substack{M\in\Gamma \\  \Re(M\tau_0)=0}}\left(\sgn(u)\tau^{n}\right)\Bigl|_{2-2k}M.
	\]
\end{lemma}

\begin{proof}
We start with the function $\widetilde{\mathcal{H}}_{1-k,n}(\tau)$. Note that the weight $2-2k$ invariant Laplace operator can be written as $\Delta_{2-2k} = -\xi_{2k}\circ \xi_{2-2k}$ and that $\xi_{2k}$ annihilates holomorphic functions. Hence, the action \eqref{eq diffops eichler} of $\xi_{2-2k}$ on Eichler integrals and the fact that $\mathcal{P}_{1-k,n}(\tau)$ is a polynomial on each connected component of $\H\setminus \bigcup_{M \in \Gamma}M(i\R_{+})$ imply that $\widetilde{\mathcal{H}}_{1-k,n}(\tau)$ is harmonic on this set. The singularities of $\widetilde{\mathcal{H}}_{1-k,n}(\tau)$ come from the sum in the locally polynomial part $\mathcal{P}_{1-k,n}(\tau)$ and are given by the stated formula by Lemma~\ref{lemma alternative polynomial}.

Now we consider the function $\mathcal{H}_{1-k,n}(\tau)$. Since the function $\tau \mapsto H_{k,k-1}(z,\tau)$ is harmonic on $\H \setminus \Gamma z$, the function $\mathcal{H}_{1-k,n}(\tau)$ is harmonic on $\H \setminus \bigcup_{M \in \Gamma}M(i\R_{+})$. To determine the singularities, we keep $\tau_{0}\in\H$ fixed and consider the function
\[
G_{\tau_0}(z, \tau) := \sum_{\substack{M\in\Gamma \\ \Re(M\tau_0) = 0}} \left(\frac{v^{2k-1}}{(z-\tau)(z-\overline{\tau})^{2k-1}}\right)\Big|_{2-2k, \tau} M.  
\]
Note that $\Re(M\tau_0)=0$ means that $\tau_0$ lies on the geodesic $M^{-1}(i\R_+)$, hence the above sum is finite (see also Figure 1 in the introduction). Furthermore, the function $G_{\tau_0}(z,\tau)$ is meromorphic in $z$ and harmonic in $\tau$ on $\H\setminus\Gamma z$. We split $\mathcal{H}_{1-k,n}(\tau)$ into 
\[
\mathcal{H}_{1-k,n}(\tau)  =\frac{(2i)^{2k-2}}{2\pi}\left(\int_0 ^\infty (H_{k,k-1}(iy,\tau) - G_{\tau_0}(iy,\tau))y^ndy + \int_0 ^\infty G_{\tau_0}(iy,\tau)y^ndy\right).
\] 
The function 
\[
z\mapsto H_{k,k-1}(z,\tau) - G_{\tau_0}(z,\tau)
\]
is meromorphic and has no singularities near $\tau_0$ and the function 
\[
\tau \mapsto \int_0 ^\infty (H_{k,k-1}(iy,\tau) - G_{\tau_0}(iy,\tau))y^ndy
\]
is harmonic in a neighborhood of $\tau_0$. For the second summand we compute for any $\tau\notin E_1$
\[
\int_0 ^\infty G_{\tau_0}(iy,\tau)y^ndy =(-i)^{n+1}\sum_{\substack{M\in\Gamma \\ \Re(M\tau_0)=0}}\left(v^{2k-1}\int_0 ^{i\infty} \frac{z^ndz}{(z-\tau)(z-\overline{\tau})^{2k-1}}\right)\Big|_{2-2k, \tau} M.  
\]
We now evaluate the inner integral for fixed $\tau$. If we shift the path of integration to the left towards $-\infty$, we pick up a residue $2\pi i\frac{\tau^n}{(2iv)^{2k-1}}$ if $\Re(\tau) <0$. If we shift towards $\infty$, we pick up a residue of $-2\pi i\frac{\tau^n}{(2iv)^{2k-1}}$ if $\Re(\tau) >0$. Thus we get
\begin{multline*}
\int_0 ^\infty G_{\tau_0}(iy,\tau)y^ndy =(-i)^{n+1}\sum_{\substack{M\in\Gamma \\ \Re(M\tau_0)=0}}\Bigl(\frac{v^{2k-1}}{2}\left(\int_{-\infty}^0-\int_0 ^\infty\right)\frac{z^ndz}{(z-\tau)(z-\overline{\tau})^{2k-1}} \\
-\sgn(u)\pi i\frac{\tau^n}{(2i)^{2k-1}}\Bigr)\Big|_{2-2k, \tau} M.  
\end{multline*}
The function
\[
\tau\mapsto \left(\int_{-\infty}^0-\int_0 ^\infty\right)\frac{z^ndz}{(z-\tau)(z-\overline{\tau})^{2k-1}}
\]
is harmonic on $\H$, so it does not contribute to the singularity. The sum over the signed terms yields the claimed singularity.
\end{proof}

We can now finish the proof of Theorem~\ref{theorem splitting}. It follows from Proposition~\ref{proposition modularity Htilde} and Lemma~\ref{jumps} that $\widetilde{\mathcal{H}}_{1-k,n}(\tau)$ is a locally harmonic Maass form of weight $2-2k$ with the same singularities as $\mathcal{H}_{1-k,n}(\tau)$, i.e., the difference $\mathcal{H}_{1-k,n}(\tau) - \widetilde{\mathcal{H}}_{1-k,n}(\tau)$ transforms like a modular form of weight $2-2k$ and is harmonic on all of $\H$. Furthermore, it follows from \eqref{eq diffops eichler} and Proposition~\ref{proposition diffops} that $\mathcal{H}_{1-k,n}(\tau) - \widetilde{\mathcal{H}}_{1-k,n}(\tau)$ is annihilated by $\xi_{2-2k}$ and $\mathcal{D}^{2k-1}$, which implies that it is a polynomial on $\H$. But the only one-periodic polynomials on $\H$ are the constant functions, and the only constant function which transforms like a modular form of non-zero weight is the constant 0 function. This shows $\mathcal{H}_{1-k,n}(\tau) = \widetilde{\mathcal{H}}_{1-k,n}(\tau)$ and concludes the proof of Theorem~\ref{theorem splitting}.

\section{The proof of Theorem~\ref{theorem rationality} and Proposition~\ref{proposition special value}}\label{section proof rationality}

Let $P \in \mathcal{Q}_d$ be a positive definite binary quadratic form and let $\tau_P \in \H$ be the associated CM point defined by $P(\tau_P,1) = 0$. For simplicity, we assume throughout this section that $\tau_P$ does not lie on any $\Gamma$-translate of the imaginary axis $i\R_+$ and leave the necessary adjustments in the general case to the reader. 

We start with the proof of Proposition~\ref{proposition special value}. We can write
\[
P(z,1) = \frac{\sqrt{|d|}}{2\Im(\tau_P)}(z-\tau_P)(z-\overline{\tau}_P),
\]
which implies
\begin{align*}
H_{k,0}(z,\tau_P) &= \sum_{M \in \Gamma}j(M,z)^{-2k}\left(\frac{(Mz - \tau_P)(Mz - \overline{\tau}_P)}{\Im(\tau_P)} \right)^{-k} \\
&= 2^{-k}|d|^{\frac{k}{2}}\sum_{M \in \Gamma}j(M,z)^{-2k}P(Mz,1)^{-k} = 2^{1-k}|d|^\frac{1-k}{2}|\overline{\Gamma}_P|\pi f_{k,P}(z).
\end{align*}
Using $(k-1)!H_{k,0}(z,\tau) = R_{2-2k,\tau}^{k-1}H_{k,k-1}(z,\tau)$ we obtain
\begin{align*}
r_n(f_{k,P}) &= \int_0^\infty f_{k,P}(iy)y^n dy  \\
&= \frac{2^{k-1}|d|^\frac{k-1}{2}}{|\overline{\Gamma}_P| \pi }\int_0^\infty H_{k,0}(iy,\tau_P)y^n dy  \\
&= \frac{2^{k-1}|d|^\frac{k-1}{2}}{(k-1)!|\overline{\Gamma}_P| \pi }\int_0^\infty R_{2-2k,\tau}^{k-1}H_{k,k-1}(iy,\tau)\big|_{\tau = \tau_P}y^n dy  \\
&=\frac{(-1)^{k-1}|d|^\frac{k-1}{2}}{2^{k-2}(k-1)!|\overline{\Gamma}_P|}R_{2-2k}^{k-1}\mathcal{H}_{1-k,n}(\tau_P).
\end{align*}
This finishes the proof of Proposition~\ref{proposition special value}.

Before we come to the proof of Theorem~\ref{theorem rationality}, we give a general formula for the periods $r_n(f_{k,P})$ of the meromorphic modular forms $f_{k,P}(z)$. 

\begin{theorem}\label{theorem rn formula}
	Assume that $\tau_P \in \H \setminus \bigcup_{M \in \Gamma}M(i\R_{+})$ and $0 < \Re(\tau_P) < 1$. Then we have the formula
	\begin{align*}
	r_n(f_{k,P}) &=  \frac{|d|^{\frac{k-1}{2}}}{|\overline{\Gamma}_P|}\bigg(\frac{(-1)^{k-1}}{2^{k-2}(k-1)!}R_{2-2k}^{k-1}\mathcal{P}_{1-k,n}(\tau_P)  \\
	&\qquad \qquad \quad +\frac{(-1)^{k}}{2^{k-2}(k-1)!}\left((-1)^{n+1}\overline{R_{2-2k}^{k-1}R_n^*(\tau_P)}+ R_{2-2k}^{k-1}R_n^*(\tau_P)\right) \\
	&\qquad \qquad \quad +(\delta_{n=0}+(-1)^k \delta_{n=2k-2})\frac{2^{k+1}(2k-2)!}{(4\pi)^{2k-1}(k-1)!}R_{2-2k}^{k-1}\mathcal{E}_{E_{2k}}(\tau_P)\bigg),
	\end{align*}
	where $R_{2-2k}^{k-1}\mathcal{P}_{1-k,n}(\tau)$ is explicitly given by
	\begin{align*}
	R_{2-2k}^{k-1}\mathcal{P}_{1-k,n}(\tau) &= (-v)^{1-k}\frac{(2k-2)!}{(k-1)!}r_n(E_{2k}) \\
	&\quad + \frac{(-i)^{n+1}}{n+1}R_{2-2k}^{k-1}B_{n+1}(\tau) + \frac{i^{n+1}}{2k-1-n}R_{2-2k}^{k-1}B_{2k-1-n}(\tau) \\
	&\quad + \frac{i^{1-n}}{2}\sum_{\substack{M \in \Gamma \\ ac > 0, \Re(M\tau) \leq 0}}\left(R_{2-2k}^{k-1}\tau^{n}-(-1)^n R_{2-2k}^{k-1}\tau^{2k-2-n} \right)\bigg|_0 M.
	\end{align*}
\end{theorem}

\begin{proof}
First, by Proposition~\ref{proposition special value} we have
\[
r_n(f_{k,P}) = \frac{(-1)^{k-1}|d|^\frac{k-1}{2}}{2^{k-2}(k-1)!|\overline{\Gamma}_P|}R_{2-2k}^{k-1}\mathcal{H}_{1-k,n}(\tau_P).
\]
Furthermore, by Theorem~\ref{theorem splitting} the locally harmonic Maass form $\mathcal{H}_{1-k,n}(\tau)$ has the splitting
\begin{align*}
\mathcal{H}_{1-k,n}(\tau) &= \mathcal{P}_{1-k,n}(\tau) + (-1)^{n+1}\frac{(2k-2)!}{(4\pi )^{2k-1}}\mathcal{E}_{R_n}(\tau) - R_{n}^*(\tau) \\
 &\quad-((-1)^{k}\delta_{n=0}+\delta_{n=2k-2})\frac{(2k-2)!}{(2\pi)^{2k-1}}\mathcal{E}_{E_{2k}}(\tau).
\end{align*}
By Proposition~\ref{proposition eichler integral relations}, we can rewrite 
\[
\frac{(2k-2)!}{(4\pi)^{2k-1}}R_{2-2k}^{k-1}\mathcal{E}_{R_n}(\tau) = -\overline{R_{2-2k}^{k-1}R_n^*(\tau)}.
\]
Finally, using the formula $R_{\kappa}v^j = (j+\kappa)v^{j-1}$ we can compute the action of the iterated raising operator on the constant $r_n(E_{2k})$ in the locally polynomial part $\mathcal{P}_{1-k,n}(\tau)$. This finishes the proof.
\end{proof}

Note that the Fourier expansion of the raised Eichler integrals appearing in Theorem~\ref{theorem rn formula} can be computed using Proposition~\ref{proposition eichler integral relations} and \eqref{raising and derivative}. In order to understand the algebraic nature of the expressions appearing in $R_{2-2k}^{k-1}\mathcal{P}_{1-k,n}(\tau_P)$, the following formula will be useful.

\begin{lemma}\label{lemma raising taupower}
	For $k \in \N$, $0 \leq \ell \leq 2k-1$, and $\tau = u+iv \in \H$ we have the formula
	\begin{align*}
	R_{2-2k}^{k-1}\tau^\ell &= (-v)^{1-k}\sum_{j=0}^{\min\{k-1,\ell\}}\binom{\ell}{j}\frac{(2k-2-j)!}{(k-1-j)!}(-2)^j\sum_{\substack{\alpha = 0 \\ \ell-\alpha \text{ even }}}^{\ell-j}\binom{\ell-j}{\alpha}u^{\alpha}(iv)^{\ell-\alpha} \\
&\quad + i\delta_{\ell=2k-1}(-1)^{k-1}2^{2k-2}(k-1)!v^k.
	\end{align*}
\end{lemma}

\begin{proof}
We first apply the formula
\begin{align}\label{raising and derivative}
R_{2-2k}^{k-1}f(\tau) = (-v)^{1-k}(k-1)!\sum_{j=0}^{k-1}\frac{(-2iv)^j}{j!}\binom{2k-2-j}{k-1}\frac{\partial^j}{\partial \tau^j}f(\tau),
\end{align}
which holds for any smooth function $f: \H \to \C$ (see equation (56) in Zagier's part of \cite{zagier123}). This yields
\begin{align}\label{eq raising taupower}
R_{2-2k}^{k-1}\tau^\ell = (-v)^{1-k}\sum_{j=0}^{\min\{k-1,\ell\}}\binom{\ell}{j}\frac{(2k-2-j)!}{(k-1-j)!}(-2iv)^j\tau^{\ell-j}.
\end{align}
On the other hand, using $R_{\kappa}v^j = (j+\kappa)v^{j-1}$ we compute
\begin{align*}
R_{2-2k}^{k-1}\tau^\ell &= R_{2-2k}^{k-1}(\overline{\tau}+2iv)^\ell = R_{2-2k}^{k-1} \sum_{j=0}^{\ell}\binom{\ell}{j}(2iv)^j\overline{\tau}^{\ell-j} \\
&= (-v)^{1-k}\sum_{j=0}^{\ell}\binom{\ell}{j}(2iv)^j\overline{\tau}^{\ell-j} (2k-2-j)(2k-3-j)\cdots (k-j).
\end{align*}
Note that $(2k-2-j)(2k-3-j)\cdots (k-j)$ equals $\frac{(2k-2-j)!}{(k-1-j)!}$ for $0 \leq j \leq k-1$, it vanishes for $k \leq j \leq 2k-2$, and it equals $(-1)^{k-1}(k-1)!$ if $j = 2k-1$, which can occur only if $\ell = 2k-1$. Hence we obtain
\begin{align*}
R_{2-2k}^{k-1}\tau^\ell &= (-v)^{1-k}\sum_{j=0}^{\min\{k-1,\ell\}}\binom{\ell}{j}\frac{(2k-2-j)!}{(k-1-j)!}(2iv)^j\overline{\tau}^{\ell-j}+\delta_{\ell=2k-1}(2i)^{2k-1}(k-1)!v^{k}.
\end{align*}
Comparing this with \eqref{eq raising taupower}, we see that
\[
R_{2-2k}^{k-1}\tau^\ell - \frac{1}{2}\delta_{\ell=2k-1}(2i)^{2k-1}(k-1)!v^k
\]
is real. Thus, expanding $\tau^{\ell-j}$ in \eqref{eq raising taupower} using the binomial theorem, we obtain
\begin{align*}
R_{2-2k}^{k-1}\tau^\ell 
&= (-v)^{1-k}\sum_{j=0}^{\min\{k-1,\ell\}}\binom{\ell}{j}\frac{(2k-2-j)!}{(k-1-j)!}(-2)^j\sum_{\alpha = 0}^{\ell-j}\binom{\ell-j}{\alpha}u^{\alpha}(iv)^{\ell-\alpha} \\
&= (-v)^{1-k}\sum_{j=0}^{\min\{k-1,\ell\}}\binom{\ell}{j}\frac{(2k-2-j)!}{(k-1-j)!}(-2)^j\sum_{\substack{\alpha = 0 \\ \ell-\alpha \text{ even }}}^{\ell-j}\binom{\ell-j}{\alpha}u^{\alpha}(iv)^{\ell-\alpha} \\
&\quad + i\delta_{\ell=2k-1}(-1)^{k-1}2^{2k-2}(k-1)!v^k.
\end{align*}
Here we used that the terms with odd $\ell-\alpha$ would be purely imaginary and hence cannot occur, apart from possibly the summand for $\alpha = 0$ in the case $\ell = 2k-1$, which is a purely imaginary multiple of $v^k$ and hence has to be equal to $i(-1)^{k-1}2^{2k-2}(k-1)!v^k$. This finishes the proof.
\end{proof}

We obtain the following rationality result for the special values $R_{2-2k}^{k-1}\mathcal{P}_{1-k,n}(\tau_P)$. 

\begin{lemma}\label{lemma P rational}
	We have
	\begin{align*}
|d|^{\frac{k-1}{2}}R_{2-2k}^{k-1}\mathcal{P}_{1-k,n}(\tau_P) \in \begin{cases}\Q, & \text{if $0 < n < 2k-2$ is odd,} \\
i\Q, & \text{if $0 < n < 2k-2$ is even.}
\end{cases}
\end{align*}
\end{lemma}

\begin{proof}
If we write $P=[a,b,c]$ then the corresponding CM point is given by 
\[
\tau_P = -\frac{b}{2a}+i\frac{\sqrt{|d|}}{2a}.
\]
Noting that $B_m(\tau) = \sum_{\ell=0}^m \binom{m}{\ell}B_{m-\ell}\tau^\ell$, we see from Lemma~\ref{lemma raising taupower} that 
\begin{align}\label{eq rational 1}
|d|^{\frac{k-1}{2}}R_{2-2k}^{k-1}B_m(\tau_P) \in \Q
\end{align}
for every $0 \leq m \leq 2k-2$, and 
\begin{align}\label{eq rational 2}
|d|^{\frac{k-1}{2}}R_{2-2k}^{k-1}B_{2k-1}(\tau_P) = C + i|d|^{\frac{k-1}{2}}(-1)^{k-1}2^{2k-2}(k-1)!v_P^k
\end{align}
for some constant $C \in \Q$. By the same lemma we see that 
\begin{align}\label{eq rational 3}
|d|^{\frac{k-1}{2}}R_{2-2k}^{k-1}\tau^\ell \in \Q
\end{align} for every $0 \leq \ell \leq 2k-2$ and every CM point $\tau$ of discriminant $d$. Note that $M\tau_P$ is a CM point of discriminant $d$ for every $M \in \Gamma$. Recall from Section~\ref{section prelims periods} that $r_n(E_{2k})$ vanishes for even $0 < n < 2k-2$ and is rational for odd $n$. Summarizing, we obtain the stated rationality result.
\end{proof}

Now we come to the proof of Theorem~\ref{theorem rationality}. First, it follows from the definition of $f_{k,P}(z)$ that 
\[
f_{k,P'}(iy) = \overline{f_{k,P}(iy)},
\]
where we put $P' = [a,-b,c]$ for $P = [a,b,c]$. This implies
\begin{align}\label{eq real imag}
r_n(f_{k,P}^{+}) = \Re(r_n(f_{k,P})), \qquad r_n(f_{k,P}^{-}) = -\Im(r_n(f_{k,P})).
\end{align}

It is now easy to see from Theorem~\ref{theorem rn formula} and Lemma~\ref{lemma P rational} that $r_n(f_{k,P})$ for $0 < n < 2k-2$ is a real number if $n$ is odd, and a purely imaginary number if $n$ is even. Therefore, \eqref{eq real imag} implies $r_n(f_{k,P}^\varepsilon) = 0$ for $0 < n < 2k-2$, where $\varepsilon$ is the sign of $(-1)^n$. This concludes the proof of the first part of Theorem~\ref{theorem rationality}.

Concerning the period $r_0(f_{k,P}^+) = \Re(r_0(f_{k,P}))$, we obtain from Theorem~\ref{theorem rn formula} and equations \eqref{eq rational 1}--\eqref{eq rational 3} the formula
\begin{align*}
&r_0(f_{k,P}^+) = -\frac{2^{k-1}|d|^{\frac{k-1}{2}}}{(2k-1)|\overline{\Gamma}_P|\zeta(2k)} \\
&\times\bigg(2v_P^k \zeta(2k) + 2^{3-2k}v_P^{1-k}\pi \zeta(2k-1)\binom{2k-2}{k-1}-\frac{4(2k-1)!\zeta(2k)}{(4\pi)^{2k-1}(k-1)!} \Re(R_{2-2k}^{k-1}\mathcal{E}_{2k} (\tau_P))\bigg).
\end{align*}
Here we used the evalution of $r_0(E_{2k})$ given in Section~\ref{section prelims periods}. Writing out the Fourier expansion of $R_{2-2k}^{k-1}\mathcal{E}_{2k}(\tau_P)$ explicitly using the formula \eqref{raising and derivative} and the well-known expansion
\[
E_{2k}(\tau) = 1+\frac{(2\pi i)^{2k}}{(2k-1)!\zeta(2k)}\sum_{n=1}^\infty \sigma_{2k-1}(n)e(n\tau),
\]
and then using the formula for the Epstein zeta function given in equation (2.4) in \cite{smartepstein} (note that the quadratic form $Q$ in \cite{smartepstein} is normalized to have to discriminant $-4$), we see that the expression in the second line in the above formula for $r_0(f_{k,P}^+)$ equals $2^{-k}|d|^{\frac{k}{2}}\zeta_P(k)$. This proves the second part of Theorem~\ref{theorem rationality}.

	In order to prove the rationality statement in the third item in Theorem~\ref{theorem rationality}, we note that the assumption $\sum_n a_n R_n(\tau) = 0$ implies that the raised Eichler integrals of the cusp forms $R_n(\tau)$ in Theorem~\ref{theorem rn formula} cancel out in the linear combination $\sum_n a_n r_n(f_{k,P})$. In particular, $\sum_n a_n r_n(f_{k,P})$ is given by a rational linear combination of the values $|d|^{\frac{k-1}{2}}R_{2-2k}^{k-1}\mathcal{P}_{1-k,n}(\tau_P)$. Now Lemma~\ref{lemma P rational} implies the third item in Theorem~\ref{theorem rationality}. This finishes the proof.

\end{document}